\documentclass[10pt]{article}

\usepackage[a4paper,margin=1in]{geometry}
\usepackage{amsmath,amssymb,amsthm,mathtools}
\usepackage{microtype}
\usepackage[skip=0.45\baselineskip plus 2pt,indent=0pt]{parskip}
\usepackage{enumitem}
\usepackage{hyperref}
\usepackage{aliascnt}
\usepackage[nameinlink,capitalise]{cleveref}

\hypersetup{
  colorlinks=true,
  linkcolor=blue,
  citecolor=blue,
  urlcolor=blue,
  pdftitle={Forcing Quasirandomness via Rooted F-Densities},
  pdfauthor={Heng Li and Xizhi Liu},
  pdfkeywords={graphons, hypergraphons, rooted densities, quasirandomness, entropy, Hoeffding decomposition, tournamentons}
}

\numberwithin{equation}{section}

\newtheorem{theorem}{Theorem}[section]
\newaliascnt{proposition}{theorem}
\newtheorem{proposition}[proposition]{Proposition}
\aliascntresetthe{proposition}
\newaliascnt{corollary}{theorem}
\newtheorem{corollary}[corollary]{Corollary}
\aliascntresetthe{corollary}
\newaliascnt{lemma}{theorem}
\newtheorem{lemma}[lemma]{Lemma}
\aliascntresetthe{lemma}

\crefname{theorem}{Theorem}{Theorems}
\Crefname{theorem}{Theorem}{Theorems}
\crefname{proposition}{Proposition}{Propositions}
\Crefname{proposition}{Proposition}{Propositions}
\crefname{corollary}{Corollary}{Corollaries}
\Crefname{corollary}{Corollary}{Corollaries}
\crefname{lemma}{Lemma}{Lemmas}
\Crefname{lemma}{Lemma}{Lemmas}

\newcommand{\bbE}{\mathbb E}
\newcommand{\R}{\mathbb R}
\newcommand{\calA}{\mathcal A}
\newcommand{\norm}[1]{\left\lVert #1\right\rVert}
\newcommand{\Ent}{\mathrm D}
\DeclareMathOperator{\Var}{Var}

\title{\Large\bf Forcing quasirandomness via rooted $F$-densities}
\author{%
Heng~Li\thanks{\scriptsize School of Mathematics, Shandong University, Jinan, China, and Extremal Combinatorics and Probability Group, Institute for Basic Science, Daejeon, South Korea. Email:~\texttt{heng.li@sdu.edu.cn}.}\and
Xizhi~Liu\thanks{\scriptsize School of Mathematical Sciences, University of Science and Technology of China, Hefei, China. Email:~\texttt{liuxizhi@ustc.edu.cn}.}}
\date{\today}

\begin{document}

\maketitle

\begin{abstract}
Let $F$ be a finite graph with at least one edge, and let $W$ be a graphon. We show that if the density of $F$ rooted at each edge is almost everywhere constant, then either $t(F,W)=0$ or $W$ is constant. For edge-transitive $F$, one rooted equation suffices. This recovers the edge-rooted triangle theorem of Reiher and Schacht. In their terminology, our result also shows that every clique is $2$-forcing, answering a question they posed. We give an explicit stability estimate when $W$ is bounded away from zero. Our proof has two steps: an entropy argument turns constant rooted densities into an additive identity for $\log W$, and a Hoeffding decomposition determines all solutions of that identity. The same method gives exact classifications and quantitative stability estimates for symmetric uniform hyperkernels, dissociated Aldous--Hoover hypergraphons, directed kernels, and tournamentons.
\end{abstract}


\section{Introduction}

Many quasirandomness results say that a small collection of subgraph-density identities has only constant graphon solutions. The theorem of Chung, Graham, and Wilson~\cite{CGW} is the classical example. Lov\'asz and Szegedy~\cite{LovaszSzegedy2006} introduced the graph-limit language in which these questions become rigidity problems for measurable functions. Borgs, Chayes, Lov\'asz, S\'os, and Vesztergombi~\cite{BCLSV1,BCLSV2} developed the metric theory, while Diaconis and Janson~\cite{DiaconisJanson} related graph limits to exchangeable random graphs. Lov\'asz's monograph~\cite{LovaszBook} gives a broad account of this theory. Closely related results include Janson's graphon formulation of quasirandomness~\cite{JansonQR}, the generalized quasirandom graphs of Lov\'asz and S\'os~\cite{LovaszSos}, and the finitely forcible graphons of Lov\'asz and Szegedy~\cite{LovaszSzegedyFF}.

We study a rooted form of this question. A \emph{graphon} is a symmetric measurable function $W\colon[0,1]^2\to[0,1]$. If $F=(V,E)$ is a finite graph, its \emph{homomorphism density} in $W$ is
\begin{equation}\label{eq:tFW-intro}
  t(F,W)\coloneqq\int_{[0,1]^V}\prod_{ij\in E}W(x_i,x_j)\prod_{i\in V}dx_i.
\end{equation}
For an edge $e=ab\in E(F)$, retain the variables at $a$ and $b$ and integrate the other vertex variables. The resulting \emph{edge-rooted density} is
\begin{equation}\label{eq:edge-rooted-intro}
  r_e^W(x,y)\coloneqq\int_{[0,1]^{V(F)\setminus\{a,b\}}}
  \prod_{ij\in E(F)}W(z_i,z_j)
  \prod_{i\in V(F)\setminus\{a,b\}}dz_i,
\end{equation}
where $z_a=x$ and $z_b=y$.

The motivating example is Problem~28 in Lov\'asz's 2008 list of problems on graph homomorphisms~\cite{LovaszProblems}. It asks for a characterization of the graphons for which the density of a triangle rooted at two adjacent vertices is constant. In the notation above, the equation is
\begin{equation}\label{eq:P28}
  W(x,y)\int_0^1W(x,z)W(y,z)\,dz=\textup{constant}
\end{equation}
for almost every $(x,y)$. For a positive constant, this equation forces $W$ to be constant. When the constant is zero, nonnegativity gives $t(K_3,W)=0$. Thus the full solution set consists of the triangle-free graphons and the constant graphons. Our first result extends this classification from the triangle to every finite graph when all edges are rooted.

\begin{theorem}\label{thm:graphon-exact-intro}
Let $F$ be a finite graph with $m=|E(F)|\ge1$, and let $W$ be a graphon. If $r_e^W$ is constant for every $e\in E(F)$, then either $t(F,W)=0$ or $W=p$ almost everywhere, where $p\coloneqq t(F,W)^{1/m}$.
\end{theorem}

The zero-density alternative is unavoidable. For example, every triangle-free graphon makes the expression in \eqref{eq:P28} identically zero. At positive density the conclusion is rigid. If $F$ is \emph{edge-transitive}, meaning that its automorphism group acts transitively on $E(F)$, then a single rooted equation propagates to every edge by automorphisms. Thus, for $F=K_3$, \Cref{thm:graphon-exact-intro} recovers Theorem~4.2 of Reiher and Schacht~\cite{ReiherSchacht} and also includes the immediate zero-density case.

We next state precisely the question that Reiher and Schacht posed. Let $(G_n)$ be a sequence of graphs with $N_n\coloneqq|V(G_n)|\to\infty$. For subsets $X_1,\ldots,X_\ell\subseteq V(G_n)$, let $K_k(X_1,\ldots,X_\ell;G_n)$ be the number of ordered injective $k$-tuples $(v_1,\ldots,v_k)$ that span a clique and satisfy $v_i\in X_i$ for $1\le i\le\ell$. The clique $K_k$ is \emph{$\ell$-forcing} if, for every $p\in(0,1]$, the condition
\[
  \sup_{X_1,\ldots,X_\ell\subseteq V(G_n)}
  \left|
    K_k(X_1,\ldots,X_\ell;G_n)
    -p^{\binom{k}{2}}N_n^{k-\ell}\prod_{i=1}^{\ell}|X_i|
  \right|
  =o(N_n^k)
\]
implies that $(G_n)$ is \emph{$p$-quasirandom}. We use the equivalent two-set formulation
\[
  \sup_{X,Y\subseteq V(G_n)}
  \left|e_{G_n}(X,Y)-p|X||Y|\right|
  =o(N_n^2),
\]
where $e_{G_n}(X,Y)$ counts ordered adjacent pairs in $X\times Y$. At the graphon level, the case $\ell=2$ says that a positive constant $K_k$-density rooted at one edge forces the graphon to be constant. Since $K_k$ is edge-transitive, \Cref{thm:graphon-exact-intro} shows that every clique is $2$-forcing. Reiher and Schacht~\cite[Section~5]{ReiherSchacht} observed that no clique is $1$-forcing and asked for the minimum $\ell$ when $k\ge4$. They also noted that $\ell=\lceil(k+1)/2\rceil$ suffices, a bound obtained independently by Hubai, Kr\'al', Parczyk, and Person~\cite{HubaiKralParczykPerson}. Thus the answer is $\ell=2$ for every $k\ge4$; see \Cref{cor:clique-forcing} for the compactness argument.

The $\ell$-forcing notion above is different from ordinary forcing. An ordinary forcing pair prescribes finitely many global numbers $t(F_i,W)$, whereas a rooted equation prescribes a two-variable function almost everywhere. Hereditary quasirandomness instead asks for the correct number of copies of a fixed graph inside every vertex subset. Simonovits and S\'os~\cite{SimonovitsSosHereditary} developed this theory, and Conlon, Fox, and Sudakov~\cite{ConlonFoxSudakovHereditary} later gave a quantitative proof without the regularity lemma. Grzesik, Kr\'al', and Pikhurko~\cite{GrzesikKralPikhurko} use rooted graphs and rooted quantum graphs to force more general quasirandom structures. Those settings are adjacent to ours, but none is the same as requiring a pointwise two-variable identity for an edge-rooted density. Moreover, an all-edge rooted system is generally stronger than a single rooted equation. The present paper solves the all-edge problem and its edge-transitive one-edge consequences; it does not claim a classification of arbitrary single-edge rooted pairs or ordinary forcing pairs.

The proof also gives a quantitative form. When $W$ is bounded away from zero, approximate constancy of all normalized rooted densities forces $\log W$ to be close in $L^2$ to a constant. The lower bound on $W$ keeps $\log W$ uniformly bounded. It is not needed in the exact theorem because positive constant rooted densities themselves imply $W\ge t(F,W)$ almost everywhere. We give the precise estimate in \Cref{thm:graph-stability}.

The same method applies to uniform hypergraphs. A symmetric $k$-uniform \emph{hyperkernel} is a symmetric measurable function $K\colon[0,1]^k\to[0,1]$. We use the \emph{dissociated}, or \emph{ergodic}, Aldous--Hoover representation associated with a deterministic dense hypergraph limit. In this convention the global mixing coordinate is absent, and the independent Bernoulli coordinate indexed by the full edge has been integrated out. Thus a $k$-uniform \emph{hypergraphon} is the conditional edge probability given the coordinates indexed by the nonempty proper subsets of an edge, represented by an $S_k$-invariant measurable function of the form
\[
  W\colon[0,1]^{\calA_k}\to[0,1],
  \qquad\text{where}\qquad
  \calA_k\coloneqq\{A\colon\emptyset\ne A\subsetneq[k]\}.
\]
This coordinate system comes from the representation theory of exchangeable arrays developed by Aldous~\cite{Aldous1981} and Kallenberg~\cite{Kallenberg2005}. Austin~\cite{AustinExchangeability} explains its connection with large graphs and hypergraphs. In dense hypergraph limit theory, related coordinate models are used by Elek and Szegedy~\cite{ElekSzegedyHypergraph} and by Zhao~\cite{ZhaoHypergraphLimits}. One local coordinate is attached to each nonempty proper subset of an edge. To recover rigidity in this model, we root all these local coordinates, not only the vertex coordinates.

For a hyperkernel, $t(H,K)$ and $r_e^K$ are defined by the same product integrals as in \eqref{eq:tFW-intro} and \eqref{eq:edge-rooted-intro}, with one factor $K((x_v)_{v\in e})$ for each hyperedge. More explicitly, if $H=(V,E)$, then
\[
  t(H,K)
  \coloneqq
  \int_{[0,1]^V}
  \prod_{f\in E}K(\mathbf x_f)
  \prod_{v\in V}dx_v,
\]
and, for $e\in E$,
\[
  r_e^K(\mathbf x_e)
  \coloneqq
  \int_{[0,1]^{V\setminus e}}
  \prod_{f\in E}K(\mathbf z_f)
  \prod_{v\in V\setminus e}dz_v,
\]
where $\mathbf x_f\coloneqq(x_v)_{v\in f}$ and $\mathbf z_f\coloneqq(z_v)_{v\in f}$. In the second formula, $z_v=x_v$ for $v\in e$, while the variables $z_v$ with $v\notin e$ are integrated out. The order of the coordinates is irrelevant because $K$ is symmetric. For a hypergraphon, the \emph{full-edge-rooted density} fixes every local coordinate indexed by a nonempty proper subset of the rooted edge. Its precise formula is given in \Cref{sec:hypergraph}.

\noindent\emph{Remark.}
In the hypergraphon model, rooting only the vertices of a hyperedge is insufficient. For $k=3$, choose $p\in(0,1)$ and a bounded nonconstant function $f\in L^\infty([0,1])$ with $\int f=0$. For every sufficiently small $\varepsilon>0$, let
\[
  W\coloneqq p+\varepsilon(f(x_{12})+f(x_{13})+f(x_{23})),
\]
so that $0\le W\le1$. If one fixes only the three vertex variables and integrates the pair variables, the vertex-rooted density of the one-edge $3$-uniform hypergraph is identically $p$, even though $W$ is not constant. Full-edge rooting retains the pair coordinates and detects this example.

The same rigidity statement holds in both hypergraph models.

\begin{theorem}\label{thm:hypergraph-intro}
Let $H$ be a finite $k$-uniform hypergraph with $k\ge2$ and $m=|E(H)|\ge1$.
\begin{enumerate}[label=\textup{(\roman*)}]
  \item If $K\colon[0,1]^k\to[0,1]$ is a symmetric $k$-uniform hyperkernel and $r_e^K$ is constant for every $e\in E(H)$, then either $t(H,K)=0$ or $K=p$ almost everywhere, where $p\coloneqq t(H,K)^{1/m}$.
  \item In the dissociated Aldous--Hoover representation above, if every full-edge-rooted density of $H$ in an $S_k$-invariant $k$-uniform hypergraphon $W$ is constant, then either $t(H,W)=0$ or $W=p$ almost everywhere, where $p\coloneqq t(H,W)^{1/m}$.
\end{enumerate}
\end{theorem}

A \emph{directed kernel} is an arbitrary measurable function $W\colon[0,1]^2\to[0,1]$. We use the one-function homomorphism-density model, which is the direct analogue of the graphon equations above but is not the most general digraphon model. Directed analogues of the usual Sidorenko and forcing problems were studied by Fox, Himwich, Mani, and Zhou~\cite{FoxHimwichManiZhou}; their conditions prescribe global homomorphism densities, whereas ours retains the variables on each rooted arc. Here symmetry is absent, and Eulerian oriented graphs lead to a family of nonconstant solutions.

For an oriented graph $D=(V,A)$, define $t(D,W)$ by the product integral in \eqref{eq:tFW-intro}, using the ordered factor $W(x_i,x_j)$ for every arc $(i,j)\in A$. Explicitly,
\[
  t(D,W)
  \coloneqq
  \int_{[0,1]^V}
  \prod_{(s,t)\in A}W(x_s,x_t)
  \prod_{v\in V}dx_v.
\]
For an arc $a=(i,j)$, the \emph{arc-rooted density} $r_a^W(x,y)$ is defined by fixing $x_i=x$ and $x_j=y$ and integrating the remaining vertex variables. We call $D$ \emph{Eulerian} if every non-isolated vertex has equal indegree and outdegree.

The absence of symmetry leads to the following classification.

\begin{theorem}\label{thm:directed-intro}
Let $D=(V,A)$ be a finite oriented graph with $m=|A|\ge1$, and let $W\colon[0,1]^2\to[0,1]$ be a directed kernel.  Suppose that $r_a^W$ is constant for every arc $a\in A$ and that $t(D,W)>0$.
\begin{enumerate}[label=\textup{(\roman*)}]
  \item If $D$ is not Eulerian, then $W=p$ almost everywhere, where $p\coloneqq t(D,W)^{1/m}$.
  \item If $D$ is Eulerian, then
  \[
    W(x,y)=p\exp(\phi(x)-\phi(y))
  \]
  for $p\coloneqq t(D,W)^{1/m}\in(0,1]$ and some bounded measurable $\phi$.
\end{enumerate}
Conversely, every kernel of this form with $0\le W\le1$ has every arc-rooted density of every Eulerian $D$ equal to $p^m$.
\end{theorem}

The vertex potentials in this family cancel on every Eulerian oriented graph. Following the tournament-limit terminology used by Th\"ornblad~\cite{ThornbladTournamentLimits}, a \emph{tournamenton} is a directed kernel satisfying
\[
  W(x,y)+W(y,x)=1
\]
for almost every $(x,y)$.

Quasirandom tournaments were introduced by Chung and Graham~\cite{ChungGrahamTournaments}. Buci\'c, Long, Shapira, and Sudakov~\cite{BucicEtAlLocalTournament} studied forcing by a single global tournament density, and Hancock et al.~\cite{HancockEtAlTournamentForcing} completed the resulting classification. Buci\'c et al. also study a local condition in which the correct count is required inside every vertex subset. This is different from fixing the two endpoints of an arc and requiring the resulting two-variable density to be constant. Noel, Ranganathan, and Simbaqueba~\cite{NoelRanganathanSimbaqueba} established forcing results under a regularity assumption, while Kr\'al' et al.~\cite{KralEtAlRegularTournaments} studied the Sidorenko property and forcing in regular tournaments. A more closely related result of Chatterjee and Bhattacharya~\cite{ChatterjeeBhattacharya} concerns a symmetrized two-point density of cyclic triangles; we compare the two conditions in \Cref{sec:directed}.

The tournamenton identity removes the nonconstant Eulerian family.

\begin{theorem}\label{thm:tournament-intro}
Let $T$ be a finite tournament with at least one arc, and let $W$ be a tournamenton. If $r_a^W$ is constant for every arc $a\in A(T)$, then either $t(T,W)=0$ or $W=1/2$ almost everywhere. If $T$ is \emph{arc-transitive}, meaning that its automorphism group acts transitively on $A(T)$, the same conclusion follows from the constancy of $r_a^W$ for a single arc.
\end{theorem}

The tournamenton statement is also quantitative when the kernel is bounded away from zero.  Nonregular tournaments inherit a linear estimate from the directed degree gap.  For regular tournaments, the entropy argument first places $\log W$ near the Eulerian gauge subspace; a quantitative three-point argument then forces $W$ close to $1/2$.  The constants and exponents in our stability estimates are explicit but not optimized.

We now describe the two main tools, with some detail for readers from combinatorics. Let $Z\coloneqq t(F,W)>0$. Normalizing the product of edge weights gives a probability measure on vertex assignments, with density
\[
  Z^{-1}\prod_{ij\in E(F)}W(x_i,x_j)
\]
with respect to product Lebesgue measure. This is the \emph{Gibbs measure} associated with $(F,W)$. Constant rooted densities say exactly that the pair of variables on every edge remains uniformly distributed under this new measure. Relative entropy, which measures the difference between two probability measures, then gives
\[
  m\int\log W-\log Z\ge0,
\]
while Jensen's inequality gives the reverse inequality. Equality and strict convexity imply that
\[
  \sum_{ij\in E(F)}\log W(x_i,x_j)
\]
is almost surely constant. A \emph{Hoeffding decomposition}, also called an ANOVA decomposition, writes a function on a product space as an orthogonal sum indexed by the coordinate sets on which the summands depend. In our setting, this makes it possible to compute the variance of the displayed edge-sum. The variance vanishes only for constant symmetric kernels; for directed kernels, it also vanishes on the Eulerian family $\phi(x)-\phi(y)$. The hypergraph proofs use the same decomposition on the vertex coordinates or on the local Aldous--Hoover coordinates.

In \Cref{sec:tools} we develop the entropy and variance reductions. \Cref{sec:graphon} proves the graphon results. \Cref{sec:hypergraph} treats symmetric hyperkernels and Aldous--Hoover hypergraphons. \Cref{sec:directed} treats directed kernels and tournamentons.

All graphs, oriented graphs, and hypergraphs are finite and simple.  All equalities between measurable functions are understood almost everywhere. Unless stated otherwise, all integrals are with respect to Lebesgue measure on the displayed unit cube, and all $L^p$ norms are taken over the natural domain.  Isolated vertices play no role and may be discarded.

\section{Entropy and variance reductions}\label{sec:tools}

This section develops the two facts used in every proof. The entropy argument converts constant rooted densities into a constant sum of logarithmic edge weights. The variance argument then shows that a constant sum of copies of a symmetric kernel can occur only when the kernel itself is constant.

\subsection{Gibbs marginals and entropy}

Write $\lambda$ for Lebesgue measure on $[0,1]$ and $\lambda^V$ for its product on $[0,1]^V$. All graphons are represented by measurable symmetric functions $W\colon[0,1]^2\to[0,1]$. If $F=(V,E)$ is a finite graph, $t(F,W)$ is defined by \eqref{eq:tFW-intro}. For $e=ab\in E$, the rooted density $r_e^W$ is defined by \eqref{eq:edge-rooted-intro}. It satisfies
\begin{equation}\label{eq:root-integral}
  \int_{[0,1]^2}r_e^W(x,y)\,dxdy=t(F,W).
\end{equation}
Hence, if $r_e^W$ is constant, then it is identically $t(F,W)$.

Suppose $Z\coloneqq t(F,W)>0$. The \emph{$F$-Gibbs measure in $W$} is the probability measure
\begin{equation}\label{eq:gibbs-graph}
  \frac{d\mu_{F,W}}{d\lambda^V}(\mathbf x)
  \coloneqq\frac1Z\prod_{ij\in E}W(x_i,x_j).
\end{equation}
The density of the marginal distribution of $(x_a,x_b)$ under $\mu_{F,W}$ is the \emph{normalized edge marginal}
\[
  q_e^W(x,y)\coloneqq\frac{r_e^W(x,y)}{Z}.
\]
Thus $r_e^W$ is constant if and only if $q_e^W\equiv1$.

The \emph{relative entropy} of a probability measure $\mu$ with density $\rho$ with respect to a probability measure $\nu$ is
\[
  \Ent(\mu\Vert\nu)\coloneqq\int \rho\log\rho\,d\nu\ge0.
\]
It is nonnegative, with equality if and only if $\rho=1$ almost everywhere. We use only this inequality and its equality case.

We first turn uniform edge marginals into an additive identity.

\begin{proposition}\label{prop:graph-entropy-reduction}
Let $F=(V,E)$ be a finite graph with $m=|E|\ge1$, let $W$ be a graphon, and put $Z\coloneqq t(F,W)$. Suppose $Z>0$ and $q_e^W\equiv1$ for every $e\in E$. Then $u\coloneqq\log W$ is bounded and
\begin{equation}\label{eq:entropy-reduction-sum}
  \sum_{ij\in E}u(x_i,x_j)=\log Z
\end{equation}
for almost every $\mathbf x\in[0,1]^V$.
\end{proposition}

\begin{proof}
Since $q_e^W\equiv1$, we have $r_e^W\equiv Z$ for every edge $e$.  Fixing $e=ab$ and using that all other factors in the rooted density are at most $1$ gives
\[
  Z=r_e^W(x,y)\le W(x,y)
\]
for almost every $(x,y)$.  Hence $W\ge Z>0$, and $u=\log W$ is bounded.

Let $\mu\coloneqq\mu_{F,W}$ be the Gibbs measure in \eqref{eq:gibbs-graph}. Since each edge marginal is uniform,
\[
  \Ent(\mu\Vert\lambda^V)
  =\sum_{e\in E}\int_{[0,1]^2}u(x,y)q_e^W(x,y)\,dxdy-\log Z
  =m\bar u-\log Z\ge0,
\]
where $\bar u\coloneqq\int_{[0,1]^2}u$. Thus $m\bar u\ge\log Z$. On the other hand, Jensen's inequality gives
\[
  \log Z
  =\log\int\exp\left(\sum_{ij\in E}u(x_i,x_j)\right)\prod_i dx_i
  \ge m\bar u.
\]
Equality holds.  Since the exponential function is strictly convex, \eqref{eq:entropy-reduction-sum} follows.
\end{proof}

\subsection{Hoeffding decomposition and variance}

We next use the two-variable case of the standard Hoeffding, or ANOVA, decomposition~\cite{Hoeffding1948}. For a symmetric function $u\in L^2([0,1]^2)$, it gives the orthogonal decomposition
\begin{equation}\label{eq:sym-decomp}
  u(x,y)=\bar u+f(x)+f(y)+g(x,y),
\end{equation}
where $\bar u\coloneqq\int_{[0,1]^2}u$, the function $f$ has mean zero, and $g$ is symmetric and satisfies $\int_0^1g(x,y)\,dy=0$ for almost every $x$. The last condition says that $g$ has no component depending on only one coordinate. Orthogonality gives
\begin{equation}\label{eq:sym-norm}
  \norm{u-\bar u}_2^2=2\norm{f}_2^2+\norm{g}_2^2.
\end{equation}
For combinatorial purposes, this plays the same role as decomposing a function into contributions supported on different coordinate sets: the constant term uses no coordinate, the two copies of $f$ use one coordinate, and $g$ uses both coordinates.

The corresponding variance calculation is the second reduction.

\begin{proposition}\label{lem:graph-var}
Let $F=(V,E)$ have $m=|E|\ge1$.  Let $u\in L^2([0,1]^2)$ be symmetric, and let
\[
  S(\mathbf X)\coloneqq\sum_{ij\in E}u(X_i,X_j),
\]
where $(X_i)_{i\in V}$ are independent uniform random variables.  Then
\begin{equation}\label{eq:graph-var}
  \norm{u-\bar u}_2^2\le \frac1m\Var(S).
\end{equation}
In particular, if $S$ is almost surely constant, then $u$ is constant almost everywhere.
\end{proposition}

\begin{proof}
Use \eqref{eq:sym-decomp}.  Then
\[
  S-m\bar u=\sum_{v\in V}d_F(v)f(X_v)+\sum_{ij\in E}g(X_i,X_j).
\]
The two sums are orthogonal.  The variables $f(X_v)$ are orthogonal for different vertices, and the variables $g(X_i,X_j)$ are orthogonal for different edges: if two edges share a vertex, degeneracy of $g$ gives zero covariance; if they are disjoint, independence gives zero covariance.  Hence
\[
  \Var(S)=\left(\sum_{v\in V}d_F(v)^2\right)\norm{f}_2^2+m\norm{g}_2^2.
\]
Since $\sum_vd_F(v)^2\ge \sum_vd_F(v)=2m$, this and \eqref{eq:sym-norm} imply \eqref{eq:graph-var}. If $S$ is constant, then \eqref{eq:graph-var} gives $\norm{u-\bar u}_2=0$.
\end{proof}

For the quantitative results, equality in Jensen's inequality must be replaced by control of its deficit.  The following elementary estimate turns the exponential deficit of a bounded centered random variable into a variance bound.

\begin{lemma}\label{lem:exp-deficit}
Let $X$ be a real random variable with $\bbE X=0$ and $|X|\le R$ almost surely.  Then
\[
  \Var(X)\le (2e^R+R^2)\log \bbE e^X.
\]
\end{lemma}

\begin{proof}
For $|x|\le R$,
\[
  e^x\ge 1+x+\frac{e^{-R}}2x^2.
\]
Taking expectations gives $\bbE e^X\ge1+e^{-R}\Var(X)/2$.  Since $\Var(X)\le R^2$, the inequality $\log(1+s)\ge s/(1+s)$ yields
\[
  \log \bbE e^X\ge
  \frac{e^{-R}\Var(X)/2}{1+e^{-R}R^2/2}.
\]
Rearranging proves the claim.
\end{proof}

\section{Graphon rigidity and stability}\label{sec:graphon}

We now combine the entropy and variance reductions. We first prove the exact result, then retain the two error terms to obtain stability.

\begin{proof}[Proof of Theorem~\ref{thm:graphon-exact-intro}]
Let $Z\coloneqq t(F,W)$. If $Z=0$, there is nothing to prove. Assume $Z>0$. By \eqref{eq:root-integral}, each constant rooted density is equal to $Z$. Thus $q_e^W\equiv1$ for every $e\in E(F)$.

By Proposition~\ref{prop:graph-entropy-reduction}, $u\coloneqq\log W$ satisfies
\[
  \sum_{ij\in E(F)}u(x_i,x_j)=\log Z
\]
for almost every $\mathbf x$. The final assertion of \Cref{lem:graph-var} shows that $u$ is constant. Hence $W$ is constant, say $W=p$. Finally, $Z=p^m$, and therefore $p=Z^{1/m}$.
\end{proof}

We can now answer the clique-forcing question stated in the introduction.

\begin{corollary}\label{cor:clique-forcing}
Every clique $K_k$ with $k\ge2$ is $2$-forcing. Consequently, for every $k\ge4$, the minimum $\ell$ in the question of Reiher and Schacht is $2$.
\end{corollary}

\begin{proof}
Fix $p\in(0,1]$, put $m\coloneqq\binom{k}{2}$, and let $(G_n)$ satisfy the two-set condition in the definition of $2$-forcing. Let $W_n$ be the empirical graphon of $G_n$, with vertex intervals $(I_v)_{v\in V(G_n)}$, and root $K_k$ at its first edge $e$. For $X,Y\subseteq V(G_n)$, put
\[
  A_X\coloneqq\bigcup_{v\in X}I_v,
  \qquad\text{and}\qquad
  A_Y\coloneqq\bigcup_{v\in Y}I_v.
\]
Since $G_n$ is loopless and $K_k$ is complete, every nonzero homomorphism from $K_k$ to $G_n$ is injective. Consequently,
\[
  \int_{A_X\times A_Y}r_e^{W_n}(x,y)\,dxdy
  =\frac{K_k(X,Y;G_n)}{N_n^k}.
\]
The two-set hypothesis therefore gives
\[
  \sup_{A,B\subseteq[0,1]}
  \left|
    \int_{A\times B}\bigl(r_e^{W_n}(x,y)-p^m\bigr)\,dxdy
  \right|
  =o(1).
\]
Indeed, the hypothesis first gives this estimate when $A$ and $B$ are unions of vertex intervals. The function $r_e^{W_n}-p^m$ is constant on each product $I_u\times I_v$. After recording the proportions of $A$ and $B$ in these intervals, the integral is a bilinear function on two finite-dimensional cubes. The maximum of this function or its negative is attained at a vertex of those cubes, which corresponds to taking $A$ and $B$ to be unions of vertex intervals.

The displayed supremum is unchanged by a measure-preserving relabeling. By compactness of the graphon space, every subsequence has a further subsequence that, after such relabelings, converges in cut norm to a graphon $W$; we continue to denote the relabeled graphons by $W_n$. For fixed measurable sets $A,B\subseteq[0,1]$, the integral
\[
  \int_{A\times B}r_e^U(x,y)\,dxdy
\]
is the homomorphism density of $K_k$ in $U$ with vertex weights $\mathbf 1_A$ and $\mathbf 1_B$ on the two roots. The weighted counting lemma~\cite[Theorem~10.23]{LovaszBook}, or its usual telescoping proof with these two vertex weights, therefore gives
\[
  \int_{A\times B}r_e^{W_n}(x,y)\,dxdy
  \longrightarrow
  \int_{A\times B}r_e^W(x,y)\,dxdy.
\]
Combining this convergence with the preceding uniform estimate yields
\[
  \int_{A\times B}\bigl(r_e^W(x,y)-p^m\bigr)\,dxdy=0
\]
for all measurable $A,B\subseteq[0,1]$. Hence $r_e^W=p^m$ almost everywhere. Since $K_k$ is edge-transitive, the same identity holds for every edge, so \Cref{thm:graphon-exact-intro} gives $W=q$ almost everywhere for some $q\in[0,1]$. The rooted identity now gives $q^m=p^m$, and hence $q=p$.

Every graphon limit of $(G_n)$ is therefore the constant graphon $p$. The equivalence between convergence to this graphon and $p$-quasirandomness, as recorded by Lov\'asz~\cite[Section~16.7.1]{LovaszBook}, now proves that $(G_n)$ is $p$-quasirandom. Thus $K_k$ is $2$-forcing. Reiher and Schacht~\cite[Section~5]{ReiherSchacht} showed that no clique is $1$-forcing, so the minimum is $2$.
\end{proof}

The exact proof used equality in both the entropy inequality and Jensen's inequality. We now keep track of the two deficits. The constants below are explicit but not optimized.

For a graph $F$ and a graphon $W$ with $t(F,W)>0$, define the \emph{total marginal error} by
\[
  \Delta_F(W)\coloneqq\sum_{e\in E(F)}\norm{q_e^W-1}_{L^1([0,1]^2)}.
\]

\begin{theorem}\label{thm:graph-stability}
Let $F$ be a finite graph with $m=|E(F)|\ge1$. Suppose $a\le W\le1$ almost everywhere, where $a\in(0,1]$, and put $M\coloneqq\log(1/a)$. Then, with $\bar u\coloneqq\int\log W$,
\begin{equation}\label{eq:graph-stability}
  \norm{\log W-\bar u}_2^2
  \le \frac1m(2e^{mM}+m^2M^2)M\Delta_F(W).
\end{equation}
Consequently, if $p\coloneqq e^{\bar u}$, then
\[
  \max\left\{\norm{W-p}_1^2,~\norm{W-p}_2^2 \right\} \le \frac1m(2e^{mM}+m^2M^2)M\Delta_F(W). 
\]
\end{theorem}

\begin{proof}
Let $u\coloneqq\log W$, $Z\coloneqq t(F,W)$, and $S\coloneqq\sum_{ij\in E(F)}u(X_i,X_j)$ for independent uniform $X_i$. Let
\[
  J\coloneqq\log Z-m\bar u\ge0
\]
be the \emph{Jensen deficit}. The Gibbs entropy identity gives
\begin{align*}
  \Ent(\mu_{F,W}\Vert\lambda^V)
  &=-J+\sum_{e\in E(F)}\int u(x,y)(q_e^W(x,y)-1)\,dxdy.
\end{align*}
Since entropy is nonnegative and $|u|\le M$,
\[
  J\le M\Delta_F(W).
\]
Now $X\coloneqq S-\bbE S$ satisfies $|X|\le mM$ and
\[
  \log\bbE e^X=J.
\]
By \Cref{lem:exp-deficit},
\[
  \Var(S)\le (2e^{mM}+m^2M^2)J.
\]
Combine this bound with \Cref{lem:graph-var} to obtain \eqref{eq:graph-stability}. Finally, $x\mapsto e^x$ is $1$-Lipschitz on $[-M,0]$, so $\norm{W-p}_2\le\norm{u-\bar u}_2$ and $\norm{W-p}_1\le\norm{W-p}_2$.
\end{proof}

\section{Hyperkernels and full hypergraphons}\label{sec:hypergraph}

We now apply the same argument to uniform hypergraphs. The entropy step is unchanged; only the variance calculation depends on the coordinate model. For this step we need a finite-product version of the Hoeffding decomposition used in \Cref{sec:tools}, with components indexed by arbitrary subsets of the coordinate set~\cite{Hoeffding1948}.

For a nonempty set $R$ and $i\in R$, a function $h\in L^2([0,1]^R)$ is \emph{degenerate in coordinate $i$} if $\int_0^1h(x_R)\,dx_i=0$ for almost every choice of the other variables.

\begin{proposition}\label{prop:finite-hoeffding}
Let $I$ be a finite set and let $u\in L^2([0,1]^I)$.  There are unique functions $u_R\in L^2([0,1]^R)$, $R\subseteq I$, such that
\[
  u(x_I)=\sum_{R\subseteq I}u_R(x_R),
\]
where $u_\emptyset\coloneqq\int u$ and every component $u_R$ with $R\ne\emptyset$ is degenerate in each coordinate in $R$. After being viewed as functions on $[0,1]^I$, the components are mutually orthogonal. In particular,
\[
  \norm{u-\int u}_2^2=\sum_{\emptyset\ne R\subseteq I}\norm{u_R}_2^2.
\]
\end{proposition}

\noindent\emph{Remark.}
For $S\subseteq I$, let $\bbE[u\mid x_S]$ denote the function obtained by averaging $u$ over the coordinates outside $S$. The components are given explicitly by
\[
  u_R\coloneqq\sum_{S\subseteq R}(-1)^{|R|-|S|}\bbE[u\mid x_S].
\]
This formula gives the degeneracy and orthogonality in the proposition. It also shows that the decomposition respects symmetries. More precisely, if a finite group $\Gamma$ acts on $I$ and $u$ is invariant under this action, then, for every $\gamma\in\Gamma$, the component indexed by $\gamma R$ is obtained from $u_R$ by relabelling coordinates. We use this equivariance below to identify components related by coordinate permutations.

\subsection{Symmetric uniform hyperkernels}

We first consider the direct $k$-variable analogue of a graphon. Symmetry identifies the Hoeffding components with the same number of coordinates, which makes the variance calculation simple.

Recall that a symmetric $k$-uniform hyperkernel is a measurable function
\[
  K\colon[0,1]^k\to[0,1]
\]
that is invariant under permutations of the coordinates.  If $H=(V,E)$ is a finite $k$-uniform hypergraph, define
\[
  t(H,K)\coloneqq\int_{[0,1]^V}\prod_{e\in E}K(x_e)\prod_{v\in V}dx_v,
\]
where $x_e$ denotes the $k$-tuple $(x_v)_{v\in e}$ in any order.  For $e\in E$, the edge-rooted density is
\begin{equation}\label{eq:hyperkernel-root}
  r_e^K(\mathbf x_e)\coloneqq\int_{[0,1]^{V\setminus e}} \prod_{f\in E}K(\mathbf z_f)\prod_{v\notin e}dz_v,
\end{equation}
where the variables on $e$ are fixed to be $\mathbf x_e$.

Applying \Cref{prop:finite-hoeffding} to a symmetric $u\in L^2([0,1]^k)$ gives unique symmetric functions $u_s\in L^2([0,1]^s)$, $1\le s\le k$, each degenerate in every coordinate, such that
\[
  u(x_1,\ldots,x_k)=\bar u+
  \sum_{s=1}^k \sum_{A\in \binom{[k]}{s}} u_s(x_A).
\]
The summands are orthogonal, so
\[
  \norm{u-\bar u}_2^2
  =\sum_{s=1}^k \binom{k}{s}\norm{u_s}_2^2.
\]
Here symmetry and equivariance identify all components indexed by coordinate sets of the same size.

The following estimate is the hyperkernel analogue of \Cref{lem:graph-var}.

\begin{proposition}\label{lem:hyperkernel-var}
Let $H=(V,E)$ be a finite $k$-uniform hypergraph with $m=|E|\ge1$.  If $u\in L^2([0,1]^k)$ is symmetric and $S_H(u)\coloneqq\sum_{e\in E}u(X_e)$,  where $(X_v)_{v\in V}$ are independent uniform variables, then
\[
  m\norm{u-\bar u}_2^2 \le \Var(S_H(u)).
\]
Consequently, if $S_H(u)$ is almost surely constant, then $u$ is constant almost everywhere.
\end{proposition}

\begin{proof}
Use the symmetric Hoeffding decomposition above. For $B\subseteq V$ with $|B|=s$, let
\[
  c_B\coloneqq\#\{e\in E\colon B\subseteq e\}.
\]
Then
\[
  S_H(u)-m\bar u
  =\sum_{s=1}^k\sum_{B\in\binom{V}{s}}c_B u_s(X_B).
\]
The terms indexed by distinct vertex sets $B$ are orthogonal by degeneracy, exactly as in \Cref{prop:finite-hoeffding}.  Therefore
\[
  \Var(S_H(u))=
  \sum_{s=1}^k\sum_{B\in\binom{V}{s}}c_B^2\norm{u_s}_2^2.
\]
For each fixed $s$, the integers $c_B$ satisfy
\[
  \sum_{B\in\binom{V}{s}}c_B=m\binom{k}{s},
  \qquad\text{and}\qquad
  \sum_{B\in\binom{V}{s}}c_B^2\ge \sum_{B\in\binom{V}{s}}c_B.
\]
Combining this with the norm identity above gives the stated variance bound. If $S_H(u)$ is constant, then the variance is zero and hence $u=\bar u$ in $L^2$.
\end{proof}

We can now combine the entropy reduction with \Cref{lem:hyperkernel-var}.

\begin{proof}[Proof of Theorem~\ref{thm:hypergraph-intro}\textup{(i)}]
Let $Z\coloneqq t(H,K)$. If $Z=0$, there is nothing to prove, so assume $Z>0$. Integrating \eqref{eq:hyperkernel-root} over the rooted variables gives $t(H,K)$, so the constant value of each $r_e^K$ is $Z$. Thus $r_e^K\equiv Z$ for every $e\in E(H)$. Since all factors in \eqref{eq:hyperkernel-root} except the rooted factor are at most $1$, we have $K\ge Z$ almost everywhere. Thus $u\coloneqq\log K$ is bounded.

Let $\mu$ be the $H$-Gibbs measure with density $Z^{-1}\prod_{e\in E(H)}K(x_e)$.  The rooted hypotheses say that every $k$-dimensional edge marginal of $\mu$ is Lebesgue measure.  Therefore
\[
  \Ent(\mu\Vert\lambda^V)=m\int_{[0,1]^k}u-\log Z\ge0.
\]
Jensen's inequality gives
\[
  \log Z
  =\log\bbE\exp\left(\sum_{e\in E(H)}u(X_e)\right)
  \ge m\int_{[0,1]^k}u.
\]
Equality holds, and the strict convexity of the exponential implies
\[
  \sum_{e\in E(H)}u(X_e)=\log Z
\]
almost surely.  By \Cref{lem:hyperkernel-var}, $u$ is constant.  Hence $K=p$ almost everywhere and $p=Z^{1/m}$.
\end{proof}

Keeping the entropy and Jensen deficits quantitative gives the corresponding stability theorem.

For a $k$-uniform hypergraph $H$ and a hyperkernel $K$ with $t(H,K)>0$, put $q_e^K\coloneqq r_e^K/t(H,K)$ and define the \emph{total marginal error} by
\[
  \Delta_H(K)\coloneqq\sum_{e\in E(H)}\norm{q_e^K-1}_{L^1([0,1]^k)}.
\]

\begin{theorem}
Let $H$ be a finite $k$-uniform hypergraph with $k\ge2$ and $m\ge1$ edges. Suppose $a\le K\le1$ almost everywhere for some $a\in(0,1]$, and put $M\coloneqq\log(1/a)$. Then
\begin{equation}\label{eq:hyperkernel-stability}
  \norm{\log K-\int\log K}_2^2
  \le \frac1m(2e^{mM}+m^2M^2)M\Delta_H(K).
\end{equation}
Consequently, if $p\coloneqq\exp(\int\log K)$, then
\[
  \max\left\{\norm{K-p}_2^2,~\norm{K-p}_1^2\right\}\le \frac1m(2e^{mM}+m^2M^2)M\Delta_H(K).
\]
\end{theorem}

\begin{proof}
Since $a>0$, put $Z\coloneqq t(H,K)>0$, $u\coloneqq\log K$, and $\bar u\coloneqq\int u$, and let
\[
  S\coloneqq\sum_{e\in E(H)}u(X_e).
\]
The entropy calculation used in \Cref{thm:graph-stability} gives the Jensen deficit bound
\[
  J\coloneqq\log Z-m\bar u\le M\Delta_H(K).
\]
For $X\coloneqq S-\bbE S$, we have $|X|\le mM$ and $\log\bbE e^X=J$.  Hence \Cref{lem:exp-deficit} gives
\[
  \Var(S)\le (2e^{mM}+m^2M^2)M\Delta_H(K).
\]
The lower bound \Cref{lem:hyperkernel-var} gives
\[
  m\norm{u-\bar u}_2^2\le \Var(S),
\]
which proves \eqref{eq:hyperkernel-stability}.  The final $L^2$ and $L^1$ conclusions follow from the same Lipschitz argument as in \Cref{thm:graph-stability}.
\end{proof}

\subsection{Full hypergraphons}

The symmetric kernel model uses only vertex coordinates. Dense hypergraph limits require additional coordinates for proper subsets of an edge. We use the dissociated Aldous--Hoover convention described in the introduction: there is no global mixing coordinate, and the independent Bernoulli coordinate indexed by the full edge has been integrated out. Thus $W$ is the conditional edge probability given all nonempty proper-subset coordinates, and it is represented by an $S_k$-invariant measurable function of the form
\[
  W\colon[0,1]^{\calA_k}\to[0,1],
  \qquad\text{where}\qquad
  \calA_k\coloneqq\{A\colon\emptyset\ne A\subsetneq[k]\}.
\]
Here $S_k$ permutes the ground set $[k]$ and hence the coordinates in $\calA_k$. More precisely, for $\pi\in S_k$ and a coordinate tuple $\mathbf x=(x_A)_{A\in\calA_k}$, put $(\pi\cdot\mathbf x)_A\coloneqq x_{\pi^{-1}A}$ for $A\in\calA_k$. Thus the $S_k$-invariance of $W$ means that $W(\pi\cdot\mathbf x)=W(\mathbf x)$ for every $\pi\in S_k$ and almost every $\mathbf x$. Also, $|\calA_k|=2^k-2$. For a finite $k$-uniform hypergraph $H=(V,E)$, choose independent uniform variables $X_B$, one for each nonempty $B\subseteq V$ with $|B|<k$. For every edge $e\in E$, choose an ordering $\eta_e\colon[k]\to e$ and write
\[
  X_e^\circ\coloneqq\bigl(X_{\eta_e(A)}\bigr)_{A\in\calA_k}.
\]
The $S_k$-invariance of $W$ makes the following definition independent of the chosen orderings:
\[
  t(H,W)\coloneqq\bbE\prod_{e\in E(H)}W(X_e^\circ).
\]
For $e\in E(H)$, the full-edge-rooted density is the function on $[0,1]^{\calA_k}$ defined by
\[
  r_e^W(\mathbf x)\coloneqq
  \bbE\left[
    \prod_{f\in E(H)}W(X_f^\circ)
    \;\middle|\;
    X_{\eta_e(A)}=x_A\ \text{for every }A\in\calA_k
  \right].
\]
Here and below, we use the canonical version obtained by fixing the listed coordinates and integrating all remaining independent coordinates. Thus all coordinates $X_B$ with $\emptyset\ne B\subsetneq e$ are fixed, not only the vertex variables $X_{\{v\}}$. We have $\int r_e^W=t(H,W)$, and, when $t(H,W)>0$, the \emph{normalized full-edge marginal} is
\[
  q_e^W\coloneqq\frac{r_e^W}{t(H,W)}.
\]

For $k=3$, for example, the local variables are $x_1,x_2,x_3,x_{12},x_{13},x_{23}$. There is no coordinate indexed by the full edge: once the variables indexed by its nonempty proper subsets have been fixed, $W$ itself gives the conditional edge probability. This is why full-edge rooting fixes six variables in the $3$-uniform case, rather than only the three vertex variables.

The variance argument below can be read as follows. Each local Hoeffding component supported on a family $R\subseteq\calA_k$ lifts, for an edge $e$, to the corresponding family of global coordinates $\eta_e(R)$. Components supported on different global coordinate families are orthogonal. If the same global family arises from different edges, the $S_k$-invariance of $u$ makes the lifted components agree, so they add rather than cancel. The next lemma formalizes this observation.

\begin{lemma}\label{lem:full-edge-hoeffding-var}
Let $H=(V,E)$ be a nonempty $k$-uniform hypergraph and let $u\in L^2([0,1]^{\calA_k})$ be invariant under the natural action of $S_k$ on $\calA_k$. For each edge $e\in E$, choose an ordering $\eta_e\colon[k]\to e$ and write $S_H(u)\coloneqq\sum_{e\in E}u(X_e^\circ)$. 
Then
\[
  |E|\,\norm{u-\bar u}_2^2 \le \Var(S_H(u)).
\]
In particular, if $S_H(u)$ is almost surely constant, then $u$ is constant almost everywhere.
\end{lemma}

\begin{proof}
Apply \Cref{prop:finite-hoeffding} to the local coordinate set $\calA_k$:
\begin{equation}\label{eq:full-edge-local-hoeffding}
  u=\bar u+
  \sum_{\emptyset\ne R\subseteq\calA_k}u_R,
\end{equation}
where $u_R$ depends only on the local coordinates in $R$ and is degenerate in each of them.  For an edge $e$ and $A\in\calA_k$, put
\[
  \eta_e(A)\coloneqq\{\eta_e(i)\colon i\in A\}.
\]
For a nonempty $R\subseteq\calA_k$, write
\[
  \eta_e(R)\coloneqq\{\eta_e(A)\colon A\in R\}.
\]
This is a family of global coordinates.  Denote by
\[
  u_{e,R}\coloneqq u_R\bigl((X_{\eta_e(A)})_{A\in R}\bigr)
\]
the corresponding lifted function.  Then
\begin{equation}\label{eq:full-edge-expanded-sum}
  S_H(u)-|E|\bar u
  =\sum_{e\in E}\sum_{\emptyset\ne R\subseteq\calA_k}u_{e,R}.
\end{equation}
Each $u_{e,R}$ is a Hoeffding-degenerate function of precisely the global coordinates in the family $\eta_e(R)$.  Therefore, applying \Cref{prop:finite-hoeffding} to the full independent coordinate system
\[
  \{X_B\colon\emptyset\ne B\subseteq V,\ |B|<k\},
\]
shows that $u_{e,R}$ and $u_{f,R'}$ are orthogonal whenever $\eta_e(R)\ne\eta_f(R')$.

It remains to check that terms supported on the same global coordinate family cannot cancel. Suppose
\[
  \eta_e(R)=\eta_f(R')\eqqcolon G.
\]
We claim that the lifted components agree almost everywhere:
\[
  u_{e,R}=u_{f,R'}.
\]
Indeed, set
\[
  I_R\coloneqq\bigcup_{A\in R}A,
  \qquad\text{and}\qquad
  I_{R'}\coloneqq\bigcup_{A'\in R'}A'.
\]
Taking unions in the equality $\eta_e(R)=\eta_f(R')$ gives
\[
  \eta_e(I_R)=\eta_f(I_{R'}).
\]
Thus
\[
  \theta\coloneqq\eta_f^{-1}\circ\eta_e\colon I_R\to I_{R'}
\]
is a well-defined bijection.  If $A\in R$, then $\eta_e(A)\in G$, so there is a unique $A'\in R'$ with $\eta_f(A')=\eta_e(A)$, and by definition this $A'$ is $\theta(A)$.  Hence $R'=\theta R$.  Extend $\theta$ arbitrarily to a permutation $\pi\in S_k$; then $\pi R=R'$ and
\[
  \eta_f(\pi A)=\eta_e(A)\qquad\text{for }A\in R.
\]
By the equivariance described after \Cref{prop:finite-hoeffding}, with the action convention above, the component $u_{\pi R}$ is the relabelled copy of $u_R$, namely
\[
  u_{\pi R}\bigl((t_{\pi^{-1}B})_{B\in\pi R}\bigr)
  =u_R\bigl((t_A)_{A\in R}\bigr)
\]
for almost every abstract tuple $(t_A)_{A\in R}$.  Take $t_A=X_{\eta_e(A)}$.  If $B=\pi A$, then $t_{\pi^{-1}B}=X_{\eta_e(A)}=X_{\eta_f(B)}$. Since $R'=\pi R$ and $u_{R'}=u_{\pi R}$, the last display gives
\[
  u_{f,R'}
  =u_{R'}\bigl((X_{\eta_f(A')})_{A'\in R'}\bigr)
  =u_R\bigl((X_{\eta_e(A)})_{A\in R}\bigr)
  =u_{e,R}.
\]
This proves the claim.

Consequently, inside a fixed global coordinate family $G$, all summands in \eqref{eq:full-edge-expanded-sum} are the same function.  Write this common function as $v_G$, and let $n_G$ be its multiplicity.

The global Hoeffding subspaces indexed by distinct nonempty coordinate families $G$ are mutually orthogonal.  Therefore
\[
\begin{aligned}
  \Var(S_H(u))
  =\sum_G n_G^2\norm{v_G}_2^2 \ge \sum_G n_G\norm{v_G}_2^2 
  =\sum_{e\in E}\sum_{\emptyset\ne R\subseteq\calA_k}\norm{u_R}_2^2 
  =|E|\,\norm{u-\bar u}_2^2,
\end{aligned}
\]
where the last equality follows from the local orthogonal decomposition \eqref{eq:full-edge-local-hoeffding}.  If $S_H(u)$ is constant, then the variance is zero, and the displayed lower bound forces $u=\bar u$ almost everywhere.
\end{proof}

\begin{proof}[Proof of Theorem~\ref{thm:hypergraph-intro}\textup{(ii)}]
Let $Z\coloneqq t(H,W)$. If $Z=0$, there is nothing to prove, so assume $Z>0$. Constancy of the full-edge-rooted density gives $r_e^W\equiv Z$ for every $e\in E(H)$. Since the other factors in the rooted integral are at most $1$, we have $W\ge Z$ almost everywhere on its full local coordinate space. Hence $u\coloneqq\log W$ is bounded.

Let $\mu$ be the $H$-Gibbs measure with density $Z^{-1}\prod_{e\in E(H)}W(X_e^\circ)$ with respect to the product Lebesgue measure, denoted here by $\lambda$. The full-edge-rooted hypotheses say that every full-edge marginal is Lebesgue measure on $[0,1]^{\calA_k}$. Therefore
\[
  \Ent(\mu\Vert\lambda)=m\int u-\log Z\ge0,
\]
while Jensen gives
\[
  \log Z=
  \log\bbE\exp\left(\sum_{e\in E(H)}u(X_e^\circ)\right)
  \ge m\int u.
\]
Equality holds, and the equality case of Jensen yields
\[
  \sum_{e\in E(H)}u(X_e^\circ)=\log Z
\]
almost surely.  By \Cref{lem:full-edge-hoeffding-var}, this forces $u$ to be constant, and hence $W=p$ almost everywhere with $p=Z^{1/m}$.
\end{proof}

Keeping the entropy and Jensen deficits gives the analogous stability statement for full hypergraphons.

For a $k$-uniform hypergraph $H$ and an $S_k$-invariant hypergraphon $W$ with $t(H,W)>0$, define the \emph{total marginal error} by
\[
  \Delta_H(W)\coloneqq\sum_{e\in E(H)}\norm{q_e^W-1}_{L^1([0,1]^{\calA_k})}.
\]

\begin{theorem}
Let $H$ be a finite $k$-uniform hypergraph with $k\ge2$ and $m=|E(H)|\ge1$, and let $W\colon[0,1]^{\calA_k}\to[0,1]$ be an $S_k$-invariant $k$-uniform hypergraphon. Suppose $a\le W\le1$ almost everywhere for some $a\in(0,1]$, and put $M\coloneqq\log(1/a)$. Then
\[
  \norm{\log W-\int\log W}_2^2
  \le \frac1m(2e^{mM}+m^2M^2)M\Delta_H(W).
\]
Consequently, if $p\coloneqq\exp(\int\log W)$, then
\[
  \max\left\{ \norm{W-p}_2^2,~\norm{W-p}_1^2 \right\} \le \frac1m(2e^{mM}+m^2M^2)M\Delta_H(W). 
\]
\end{theorem}

\begin{proof}
Put $Z\coloneqq t(H,W)>0$, $u\coloneqq\log W$, and $\bar u\coloneqq\int u$, and let
\[
  S\coloneqq\sum_{e\in E(H)}u(X_e^\circ).
\]
Let $\mu$ be the corresponding $H$-Gibbs measure, and denote the underlying product Lebesgue measure by $\lambda$.
The full-edge Gibbs entropy identity gives, exactly as in the graphon proof,
\[
  \Ent(\mu\Vert\lambda)
  =-J+\sum_{e\in E(H)}\int u(q_e^W-1),
  \qquad\text{where}\qquad
  J\coloneqq\log Z-m\bar u.
\]
Since entropy is nonnegative and $|u|\le M$, we have
\[
  J\le M\Delta_H(W).
\]
For $X\coloneqq S-\bbE S$, we have $|X|\le mM$ and $\log\bbE e^X=J$.  Hence \Cref{lem:exp-deficit} gives
\[
  \Var(S)\le(2e^{mM}+m^2M^2)M\Delta_H(W).
\]
Finally, \Cref{lem:full-edge-hoeffding-var} gives
\[
  m\norm{u-\bar u}_2^2\le\Var(S),
\]
and the stated logarithmic stability bound follows.  The final $L^2$ and $L^1$ conclusions follow because the exponential map is $1$-Lipschitz on $[-M,0]$.
\end{proof}

\section{Directed kernels and tournamentons}\label{sec:directed}

For a nonsymmetric kernel, the one-coordinate terms in the Hoeffding decomposition need not agree. This produces the Eulerian family in the directed problem. We first classify that family and prove stability away from it. We then impose the tournamenton identity and recover rigidity.

\subsection{Directed kernels}

A finite \emph{oriented graph} $D=(V,A)$ has no loops, repeated arcs, or pairs of opposite arcs. A directed kernel is a measurable function $W\colon[0,1]^2\to[0,1]$, not assumed symmetric. We use this one-function model; it does not cover general digraphons that allow both orientations of the same pair. Isolated vertices do not affect the density equations. Define
\begin{equation}\label{eq:digraph-density}
  t(D,W)\coloneqq\int_{[0,1]^V}\prod_{(i,j)\in A}W(x_i,x_j)\prod_{i\in V}dx_i.
\end{equation}
For an arc $a=(i,j)$, the arc-rooted density $r_a^W(x,y)$ is obtained by fixing $x_i=x$ and $x_j=y$ and integrating the remaining vertex variables. When $t(D,W)>0$, the \emph{normalized arc marginal} is
\[
  q_a^W(x,y)\coloneqq\frac{r_a^W(x,y)}{t(D,W)}.
\]

The two-variable Hoeffding decomposition~\cite{Hoeffding1948} does not require symmetry. Every $u\in L^2([0,1]^2)$ can be written uniquely as
\begin{equation}\label{eq:directed-decomp}
  u(x,y)=\bar u+f(x)+g(y)+h(x,y),
\end{equation}
where $\bar u\coloneqq\int u$, the functions $f$ and $g$ have mean zero, and $h$ has zero integral in each coordinate. The four terms are orthogonal, so
\[
  \norm{u-\bar u}_2^2=\norm{f}_2^2+\norm{g}_2^2+\norm{h}_2^2.
\]
Explicitly, $f(x)\coloneqq\int u(x,y)\,dy-\bar u$ and $g(y)\coloneqq\int u(x,y)\,dx-\bar u$. These formulas are the continuous analogue of subtracting row and column averages from a matrix.

Write $d^+(v)$ and $d^-(v)$ for the outdegree and indegree of $v$. Recall that $D$ is Eulerian if $d^+(v)=d^-(v)$ at every non-isolated vertex. For later use, define the \emph{degree gap}
\begin{equation}\label{eq:gamma-D-def}
  \gamma_D\coloneqq\inf_{(\alpha,\beta)\ne(0,0)}
  \frac{\sum_{v\in V}(d^+(v)\alpha+d^-(v)\beta)^2}{\alpha^2+\beta^2}.
\end{equation}
We have $\gamma_D>0$ exactly when $D$ is not Eulerian. Indeed, $(1,-1)$ gives zero when $D$ is Eulerian. Conversely, if the infimum is zero, all nonzero degree vectors $(d^+(v),d^-(v))$ lie on one line. Their sum is $(m,m)$, so this line is spanned by $(1,1)$ and $D$ is Eulerian.

We now classify the kernels whose copies have constant sum on an oriented graph.

\begin{proposition}\label{lem:directed-pairsum}
Let $D=(V,A)$ be a finite oriented graph with $m=|A|\ge1$, and let $u\in L^2([0,1]^2)$.  If
\[
  \sum_{(i,j)\in A}u(x_i,x_j)=C
\]
for almost every $\mathbf x\in[0,1]^V$, then:
\begin{enumerate}[label=\textup{(\roman*)}]
  \item if $D$ is not Eulerian, then $u$ is constant almost everywhere;
  \item if $D$ is Eulerian, then $u(x,y)=c+\phi(x)-\phi(y)$ for some $c\in\R$ and $\phi\in L^2([0,1])$.
\end{enumerate}
\end{proposition}

\begin{proof}
Use \eqref{eq:directed-decomp}, let $X_v$ be independent uniform variables, and put
\[
  S_D\coloneqq\sum_{(i,j)\in A}u(X_i,X_j).
\]
Then $S_D-m\bar u$ is equal to
\[
  \sum_{v\in V}(d^+(v)f(X_v)+d^-(v)g(X_v))
  +\sum_{(i,j)\in A}h(X_i,X_j).
\]
The two displayed sums are orthogonal.  The $h$-terms are mutually orthogonal because $D$ has no opposite arcs: disjoint arcs are independent, arcs sharing one endpoint have zero covariance by row or column degeneracy of $h$.  Thus
\begin{equation}\label{eq:directed-var}
  \Var(S_D) = \sum_{v\in V}\norm{d^+(v)f+d^-(v)g}_2^2+m\norm{h}_2^2.
\end{equation}
The hypothesis makes this variance zero. Hence $h=0$ and $d^+(v)f+d^-(v)g=0$ for every non-isolated $v$. If $D$ is not Eulerian, then $\gamma_D>0$. Applying the defining inequality for $\gamma_D$ pointwise to $(\alpha,\beta)=(f(x),g(x))$ and integrating gives $f=g=0$. Thus $u$ is constant. If $D$ is Eulerian, the vertex equations reduce to $d(v)(f+g)=0$. Since some vertex has positive degree, $f+g=0$, and therefore $u=c+f(x)-f(y)$.
\end{proof}

\begin{proof}[Proof of Theorem~\ref{thm:directed-intro}]
Let $Z\coloneqq t(D,W)>0$. Integrating any arc-rooted density gives $Z$, so $r_a^W\equiv Z$ for every arc. The same argument as in the graphon case gives $W\ge Z>0$. Put $u\coloneqq\log W$. The Gibbs entropy and Jensen inequality again force
\[
  \sum_{(i,j)\in A}u(x_i,x_j)=\log Z
\]
almost everywhere.  Apply \Cref{lem:directed-pairsum}.  This gives either $u$ constant or $u=c+\phi(x)-\phi(y)$. In the Eulerian case put $p\coloneqq e^c$. Since $u\in L^\infty$, the difference $\phi(x)-\phi(y)$ is essentially bounded.  More explicitly, if $|\phi(x)-\phi(y)|\le M$ for almost every $(x,y)$, then Fubini gives a point $y_0$ such that $|\phi(x)-\phi(y_0)|\le M$ for almost every $x$.  After changing $\phi$ on a null set and subtracting a constant, we may take $\phi$ bounded and $\int\phi=0$.  Then $c=\int u=\int\log W\le0$, so $p=e^c\le1$.  In the Eulerian case the gauge terms cancel in the template sum, and hence $Z=p^m$; in the non-Eulerian case the same identity is immediate from the constant form of $W$.  Thus in both cases $p=Z^{1/m}$. Exponentiating proves the stated alternatives.

Conversely, if $D$ is Eulerian and $W(x,y)=p\exp(\phi(x)-\phi(y))$, then for every assignment $\mathbf x$,
\[
  \prod_{(i,j)\in A}W(x_i,x_j)
  =p^m\prod_{v\in V}\exp((d^+(v)-d^-(v))\phi(x_v))=p^m.
\]
Thus every rooted density is identically $p^m$.
\end{proof}

\noindent\emph{Remark.}
The function $\phi$ is determined only up to an additive constant, since replacing $\phi$ by $\phi+c$ does not change $W$. The Eulerian alternative includes nonconstant kernels: if $C_\ell^{\to}$ is a directed cycle and $\phi$ is bounded and nonconstant, then, for all sufficiently small $p>0$, the kernel $W(x,y)=p\exp(\phi(x)-\phi(y))$ takes values in $[0,1]$, is nonconstant, and has every arc-rooted $C_\ell^{\to}$-density equal to $p^\ell$.

For a non-Eulerian template, the positive degree gap controls both one-coordinate terms in the directed Hoeffding decomposition. Combining this observation with the entropy-deficit estimate gives the following theorem.

For an oriented graph $D$ and a directed kernel $W$ with $t(D,W)>0$, define the \emph{total marginal error} by
\[
  \Delta_D(W)\coloneqq\sum_{b\in A(D)}\norm{q_b^W-1}_{L^1([0,1]^2)}.
\]

\begin{theorem}\label{thm:directed-stability}
Let $D$ be a finite non-Eulerian oriented graph with $m\ge1$ arcs, and let $\gamma_D$ be defined by \eqref{eq:gamma-D-def}. Suppose $0<\alpha\le W\le1$ and put $M\coloneqq\log(1/\alpha)$. Then $\gamma_D>0$ and, with $\bar u\coloneqq\int\log W$,
\begin{equation}\label{eq:directed-stability}
  \norm{\log W-\bar u}_2^2
  \le \max\left\{\frac1{\gamma_D},~\frac1m\right\}
  (2e^{mM}+m^2M^2)M\Delta_D(W).
\end{equation}
\end{theorem}

\begin{proof}
The discussion after \eqref{eq:gamma-D-def} gives $\gamma_D>0$. Since $W\ge\alpha>0$, put $Z\coloneqq t(D,W)>0$ and $u\coloneqq\log W$. Then $-M\le u\le0$. Let
\[
  S\coloneqq\sum_{(i,j)\in A(D)}u(X_i,X_j),
  \qquad\text{and}\qquad
  \bar u\coloneqq\int_{[0,1]^2}u.
\]
Let $\mu$ be the $D$-Gibbs measure with density $Z^{-1}e^S$ with respect to the product Lebesgue measure, denoted here by $\lambda$. Its entropy is
\[
  \Ent(\mu\Vert\lambda)
  =\sum_{b\in A(D)}\int_{[0,1]^2}u(x,y)q_b^W(x,y)\,dxdy-\log Z\ge0.
\]
Define
\[
  J\coloneqq\log Z-m\bar u.
\]
Jensen's inequality gives $J\ge0$.  The entropy inequality gives
\[
  J\le\sum_{b\in A(D)}\int_{[0,1]^2}u(x,y)(q_b^W(x,y)-1)\,dxdy
  \le M\Delta_D(W),
\]
because $|u|\le M$.  For $X\coloneqq S-m\bar u$, we have $\bbE X=0$, $|X|\le mM$, and
\[
  \log\bbE e^X=\log Z-m\bar u=J.
\]
By \Cref{lem:exp-deficit},
\[
  \Var(S)=\Var(X)\le(2e^{mM}+m^2M^2)M\Delta_D(W).
\]
Use the directed decomposition $u=\bar u+f(x)+g(y)+h(x,y)$.  Formula \eqref{eq:directed-var} and the variational definition \eqref{eq:gamma-D-def} yield
\[
  \Var(S)
  \ge \gamma_D(\norm{f}_2^2+\norm{g}_2^2)+m\norm{h}_2^2.
\]
Together with the norm identity following \eqref{eq:directed-decomp}, this proves \eqref{eq:directed-stability}.
\end{proof}

\subsection{Tournamentons}

The directed theorem identifies the Eulerian family as the only positive-density obstruction. Tournamentons satisfy an additional identity that forces every member of this family to be constant.

A \emph{tournamenton} is a measurable function $W\colon[0,1]^2\to[0,1]$ such that
\begin{equation}\label{eq:tournamenton}
  W(x,y)+W(y,x)=1
\end{equation}
for almost every $(x,y)$.  It is the dense limit analogue of a tournament, i.e. an orientation of a complete graph.  If $T=(V,A)$ is a finite tournament, define $t(T,W)$ by the directed formula \eqref{eq:digraph-density}.

\newpage

\noindent\emph{Remark.}
If $D$ is non-Eulerian and $W$ is a tournamenton satisfying the hypotheses of \Cref{thm:directed-stability}, then that theorem gives a linear stability estimate toward the uniform tournamenton. Indeed, put $p\coloneqq e^{\bar u}$, where $\bar u\coloneqq\int\log W$. Since the exponential function is $1$-Lipschitz on $[-M,0]$ and \eqref{eq:tournamenton} gives $\int W=1/2$, we have
\[
  \norm{W-1/2}_2
  \le \norm{W-p}_2+|p-1/2|
  \le 2\norm{W-p}_2
  \le 2\norm{\log W-\bar u}_2.
\]
Consequently,
\[
  \norm{W-1/2}_2^2
  \le 4\max\left\{\frac1{\gamma_D},~\frac1m\right\}
  (2e^{mM}+m^2M^2)M\Delta_D(W).
\]

For $\eta\in[1/2,1]$, the \emph{Condorcet tournamenton} is
\[
  W_\eta(x,y)\coloneqq
  \begin{cases}
    \eta, & x<y,\\
    1-\eta, & x>y,\\
    1/2, & x=y.
  \end{cases}
\]
Two tournamentons $U$ and $W$ are \emph{weakly equivalent} if $t(T,U)=t(T,W)$ for every finite tournament $T$; equivalently, they represent the same tournament limit.

For comparison with recent work, define the \emph{oriented two-point cyclic-triangle density}
\[
  c_W(x,y)\coloneqq W(x,y)\int_0^1W(y,z)W(z,x)\,dz.
\]
Chatterjee and Bhattacharya~\cite[Theorem~2.9]{ChatterjeeBhattacharya} call a tournamenton \emph{triangle-coregular} when the symmetrized quantity $\frac12(c_W(x,y)+c_W(y,x))$ is constant, and they prove that such tournamentons are weakly equivalent to Condorcet tournamentons. Requiring the oriented quantity $c_W(x,y)$ itself to be constant is stronger and, by \Cref{thm:tournament-intro}, forces the uniform tournamenton at positive cyclic-triangle density rather than the full Condorcet family.

The following lemma shows that the tournamenton identity leaves no nonconstant gauge kernel.

\begin{lemma}\label{lem:tournament-gauge}
Let $W$ be a tournamenton.  If $W(x,y)=p\exp(\phi(x)-\phi(y))$ for some $p>0$ and measurable $\phi$, then $W=1/2$ almost everywhere.
\end{lemma}

\begin{proof}
The tournament identity gives
\[
  p\left(e^{\phi(x)-\phi(y)}+e^{\phi(y)-\phi(x)}\right)=1
\]
for almost every $(x,y)$. The function $a\mapsto e^a+e^{-a}$ is even and strictly increasing on $[0,\infty)$, so this identity implies that $|\phi(x)-\phi(y)|$ is equal to a constant $c$ for almost every $(x,y)$. By Fubini, this pairwise relation holds simultaneously for the three pairs $(X,Y)$, $(Y,Z)$, and $(Z,X)$ for almost every triple $(X,Y,Z)$. If $c>0$, then the three differences
\[
  \phi(X)-\phi(Y),\quad \phi(Y)-\phi(Z),\quad\text{and}\quad \phi(Z)-\phi(X)
\]
would each be equal to either $c$ or $-c$, but their sum is identically zero.  No sum of three numbers from $\{c,-c\}$ is zero when $c>0$.  Therefore $c=0$, so $\phi$ is constant almost everywhere.  Then $W$ is constant, and \eqref{eq:tournamenton} gives $W=1/2$.
\end{proof}

For stability, we first measure the distance to the \emph{Eulerian gauge subspace}
\[
  \mathcal G\coloneqq\{c+\phi(x)-\phi(y)\colon c\in\R,\ \phi\in L^2([0,1])\}.
\]
This subspace is closed: after normalizing $\int\phi=0$, the map $\phi\mapsto\phi(x)-\phi(y)$ has squared norm $2\norm{\phi}_2^2$, and its range is orthogonal to the constants.

The quantitative version controls the distance from a tournamenton to the gauge subspace.

\begin{lemma}\label{lem:tournament-gauge-stability}
Let $W$ be a tournamenton such that $0<\alpha\le W\le1$ almost everywhere, put $M\coloneqq\log(1/\alpha)$, and let $\rho\coloneqq\operatorname{dist}_{L^2}(\log W,\mathcal G)$. 
Then
\begin{equation}\label{eq:quantitative-gauge-killing}
  \norm{W-1/2}_2^2
  \le \left(8\cosh^2 M+\frac M8\right)\rho.
\end{equation}
\end{lemma}

\begin{proof}
Put $u\coloneqq\log W$ and use the directed Hoeffding decomposition
\[
  u(x,y)=\bar u+f(x)+g(y)+h(x,y)
\]
from \eqref{eq:directed-decomp}. The orthogonal projection of $u$ onto $\mathcal G$ is
\[
  u_0(x,y)\coloneqq\bar u+\phi(x)-\phi(y),
  \qquad\text{where}\qquad
  \phi\coloneqq\frac{f-g}{2}.
\]
Here $\int\phi=0$ by the normalization in the Hoeffding decomposition. Indeed,
\[
  u-u_0
  =\frac{f+g}{2}(x)+\frac{f+g}{2}(y)+h(x,y),
\]
and this function is orthogonal to every constant and every function of the form $\psi(x)-\psi(y)$.  In particular,
\[
  \norm{u-u_0}_2=\rho.
\]

Define
\[
  s(x,y)\coloneqq\frac{u(x,y)+u(y,x)}2,
  \qquad\text{and}\qquad
  a(x,y)\coloneqq\frac{u(x,y)-u(y,x)}2,
\]
and put $b(x,y)\coloneqq\phi(x)-\phi(y)$. Symmetrization and antisymmetrization are orthogonal contractions on $L^2([0,1]^2)$, so
\[
  \norm{s-\bar u}_2\le\rho,
  \qquad\text{and}\qquad
  \norm{a-b}_2\le\rho.
\]
Since $W$ is a tournamenton,
\[
  e^{s(x,y)+a(x,y)}+e^{s(x,y)-a(x,y)}=1.
\]
Thus, for
\[
  H(t)\coloneqq\log(2\cosh t),
\]
we have
\[
  s=-H(a),
  \qquad\text{and}\qquad
  2W-1=\tanh a.
\]
Moreover, $-M\le u\le0$, and the same bounds hold for $u(y,x)$; hence $|a|\le M/2$ almost everywhere.

The exact gauge difference $b$ satisfies the cocycle identity $b(X,Y)+b(Y,Z)+b(Z,X)=0$. The projection estimate says that $a$ is close to $b$, while the tournament identity says that $H(a)$ is close to a constant. The following elementary estimate converts these two facts into $L^2$ control of $a$.

We use the following elementary three-point estimate. If $|x|,|y|,|z|\le R$ and $\delta\coloneqq x+y+z$, then
\[
  x^2+y^2+z^2
  \le C_R\sum_{\rm cyc}|H(x)-H(y)|
     +(2C_R+R)|\delta|,
  \qquad\text{where}\qquad
  C_R\coloneqq8\cosh^2(2R).
\]
To prove it, first suppose that $\delta=0$.  In this case we prove the slightly sharper bound with $8\cosh^2 R$ in place of $C_R$.  Let $A$ and $C$ be, respectively, the largest and smallest of $|x|,|y|,|z|$.  After permuting the variables and, if necessary, changing all three signs, the triple has the form $(p,q,-p-q)$ with $p,q\ge0$.  Thus $A=p+q$ and $C=\min\{p,q\}\le A/2$.  Since
\[
  \tanh t=\int_0^t\operatorname{sech}^2 r\,dr
  \ge t\operatorname{sech}^2 R
  \qquad\text{for }0\le t\le R,
\]
we obtain
\[
  H(A)-H(C)
  \ge\frac12\operatorname{sech}^2 R(A^2-C^2)
  \ge\frac38\operatorname{sech}^2 R A^2.
\]
As $x^2+y^2+z^2\le3A^2$ and $H$ is even, this proves the asserted sharper zero-sum bound.  In general, set
\[
  x'\coloneqq x-\frac\delta3,
  \qquad
  y'\coloneqq y-\frac\delta3,
  \qquad\text{and}\qquad
  z'\coloneqq z-\frac\delta3.
\]
Then $x'+y'+z'=0$, all three primed variables have absolute value at most $2R$, and
\[
  x^2+y^2+z^2=x'^2+y'^2+z'^2+\frac{\delta^2}{3}.
\]
Applying the zero-sum bound with parameter $2R$ gives the coefficient $C_R=8\cosh^2(2R)$.  The function $H$ is $1$-Lipschitz. Consequently, replacing the three unprimed variables by the primed ones increases the sum of the three pairwise $H$-differences by at most $2|\delta|$.  Finally, $\delta^2/3\le R|\delta|$ because $|\delta|\le3R$. This proves the three-point estimate.

Apply the estimate with $R=M/2$ to
\[
  x\coloneqq a(X,Y),
  \qquad
  y\coloneqq a(Y,Z),
  \qquad\text{and}\qquad
  z\coloneqq a(Z,X),
\]
where $X,Y,Z$ are independent uniform random variables. Because $b(X,Y)+b(Y,Z)+b(Z,X)=0$, the projection bounds above and Minkowski's inequality give
\[
  \norm{a(X,Y)+a(Y,Z)+a(Z,X)}_{L^1}
  \le3\rho.
\]
The tournamenton identity and the projection bounds also give
\[
  \norm{H(a)+\bar u}_2\le\rho.
\]
It follows that the expectation of each of the three pairwise $H$-differences in the three-point estimate is at most $2\rho$. Integrating that inequality, and using $C_{M/2}=8\cosh^2 M$, yields
\[
  3\norm{a}_2^2
  \le6C_{M/2}\rho+3(2C_{M/2}+M/2)\rho,
\]
and hence
\[
  \norm{a}_2^2
  \le(4C_{M/2}+M/2)\rho.
\]
Finally, $2W-1=\tanh a$ and $|\tanh t|\le|t|$ imply
\[
  \norm{W-1/2}_2^2
  =\frac14\norm{\tanh a}_2^2
  \le\frac14\norm{a}_2^2.
\]
The preceding bound on $\norm{a}_2^2$ now gives \eqref{eq:quantitative-gauge-killing}.
\end{proof}

\begin{proof}[Proof of Theorem~\ref{thm:tournament-intro}]
If $t(T,W)=0$, there is nothing to prove. Assume $t(T,W)>0$ and apply \Cref{thm:directed-intro} to the oriented graph underlying $T$. If $T$ is not Eulerian, the directed theorem gives that $W$ is constant, hence $W=1/2$ by \eqref{eq:tournamenton}. If $T$ is Eulerian, the directed theorem gives the gauge form, and \Cref{lem:tournament-gauge} again gives $W=1/2$.
\end{proof}

For regular tournaments, the quantitative gauge estimate gives a weaker exponent.

\begin{corollary}
Let $T$ be a regular finite tournament with $m\ge1$ arcs. Suppose $W$ is a tournamenton with $0<\alpha\le W\le1$ and put $M\coloneqq\log(1/\alpha)$. Then
\begin{equation}\label{eq:regular-tournament-stability}
  \norm{W-1/2}_2^2
  \le \left(8\cosh^2 M+\frac M8\right)
  \left(
    \frac{(2e^{mM}+m^2M^2)M}{m}\,\Delta_T(W)
  \right)^{1/2}.
\end{equation}
\end{corollary}

\begin{proof}
Since $W\ge\alpha$, we have $t(T,W)\ge\alpha^m>0$. Put $u\coloneqq\log W$ and let
\[
  u(x,y)=\bar u+f(x)+g(y)+h(x,y)
\]
be its directed Hoeffding decomposition. As in the proof of \Cref{lem:tournament-gauge-stability}, the squared distance from $u$ to $\mathcal G$ is
\[
  \rho^2
  =\frac12\norm{f+g}_2^2+\norm h_2^2.
\]

Put $d(v)\coloneqq d^+(v)=d^-(v)$, where the second equality holds because $T$ is regular. For
\[
  S\coloneqq\sum_{(i,j)\in A(T)}u(X_i,X_j),
\]
the directed variance identity \eqref{eq:directed-var} becomes
\[
  \Var(S)
  =\left(\sum_{v\in V(T)}d(v)^2\right)\norm{f+g}_2^2
    +m\norm h_2^2.
\]
The degrees $d(v)$ are nonnegative integers and $\sum_vd(v)=m$, so $\sum_vd(v)^2\ge m$. The last two displays therefore give
\[
  m\rho^2\le\Var(S).
\]

The entropy and exponential-deficit calculation in the proof of \Cref{thm:directed-stability} does not use non-Eulerianness.  It gives
\[
  \Var(S)
  \le(2e^{mM}+m^2M^2)M\Delta_T(W).
\]
Consequently,
\[
  \rho
  \le\left(
    \frac{(2e^{mM}+m^2M^2)M}{m}\Delta_T(W)
  \right)^{1/2}.
\]
Apply \Cref{lem:tournament-gauge-stability} to obtain \eqref{eq:regular-tournament-stability}.
\end{proof}

The same proof, with the same constants, applies to every Eulerian oriented graph $D$ when the kernel also satisfies the tournamenton identity. The constants and the H\"older exponent in \eqref{eq:regular-tournament-stability} are not optimized.

\section*{Acknowledgements}

H. L. was supported by the National Natural Science Foundation of China (12501487), the China Scholarship Council, and the Institute for Basic Science (IBS-R029-C4). X.L. was supported by the Excellent Young Talents Program (Overseas) of the National Natural Science Foundation of China.

\section*{Declaration on the use of AI}

The authors used generative AI tools to assist in discussing proof strategies, checking proofs, and improving exposition. 

\bibliographystyle{abbrv}
\bibliography{rooted_density_rigidity}

@article {Aldous1981,
    AUTHOR = {Aldous, David J.},
     TITLE = {Representations for partially exchangeable arrays of random variables},
   JOURNAL = {J. Multivariate Anal.},
  FJOURNAL = {Journal of Multivariate Analysis},
    VOLUME = {11},
      YEAR = {1981},
    NUMBER = {4},
     PAGES = {581--598},
      ISSN = {0047-259X},
  MRNUMBER = {0637937},
       DOI = {10.1016/0047-259X(81)90099-3},
       URL = {https://doi.org/10.1016/0047-259X(81)90099-3},
}

@article {AustinExchangeability,
    AUTHOR = {Austin, Tim},
     TITLE = {On exchangeable random variables and the statistics of large graphs and hypergraphs},
   JOURNAL = {Probab. Surv.},
  FJOURNAL = {Probability Surveys},
    VOLUME = {5},
      YEAR = {2008},
     PAGES = {80--145},
      ISSN = {1549-5787},
  MRNUMBER = {2426176},
       DOI = {10.1214/08-PS124},
       URL = {https://doi.org/10.1214/08-PS124},
}

@article {BCLSV1,
    AUTHOR = {Borgs, Christian and Chayes, Jennifer T. and Lov{\'a}sz, L{\'a}szl{\'o} and S{\'o}s, Vera T. and Vesztergombi, Katalin},
     TITLE = {Convergent sequences of dense graphs. {I}. {S}ubgraph frequencies, metric properties and testing},
   JOURNAL = {Adv. Math.},
  FJOURNAL = {Advances in Mathematics},
    VOLUME = {219},
      YEAR = {2008},
    NUMBER = {6},
     PAGES = {1801--1851},
      ISSN = {0001-8708},
  MRNUMBER = {2455626},
       DOI = {10.1016/j.aim.2008.07.008},
       URL = {https://doi.org/10.1016/j.aim.2008.07.008},
}

@article {BCLSV2,
    AUTHOR = {Borgs, Christian and Chayes, Jennifer T. and Lov{\'a}sz, L{\'a}szl{\'o} and S{\'o}s, Vera T. and Vesztergombi, Katalin},
     TITLE = {Convergent sequences of dense graphs. {II}. {M}ultiway cuts and statistical physics},
   JOURNAL = {Ann. of Math. (2)},
  FJOURNAL = {Annals of Mathematics. Second Series},
    VOLUME = {176},
      YEAR = {2012},
    NUMBER = {1},
     PAGES = {151--219},
      ISSN = {0003-486X},
  MRNUMBER = {2925382},
       DOI = {10.4007/annals.2012.176.1.2},
       URL = {https://doi.org/10.4007/annals.2012.176.1.2},
}

@article {BucicEtAlLocalTournament,
    AUTHOR = {Buci{\'c}, Matija and Long, Eoin and Shapira, Asaf and Sudakov, Benny},
     TITLE = {Tournament quasirandomness from local counting},
   JOURNAL = {Combinatorica},
  FJOURNAL = {Combinatorica. An International Journal on Combinatorics and the Theory of Computing},
    VOLUME = {41},
      YEAR = {2021},
    NUMBER = {2},
     PAGES = {175--208},
      ISSN = {0209-9683},
  MRNUMBER = {4253423},
       DOI = {10.1007/s00493-020-4371-y},
       URL = {https://doi.org/10.1007/s00493-020-4371-y},
}

@article {CGW,
    AUTHOR = {Chung, Fan R. K. and Graham, Ronald L. and Wilson, Richard M.},
     TITLE = {Quasi-random graphs},
   JOURNAL = {Combinatorica},
  FJOURNAL = {Combinatorica. An International Journal on Combinatorics and the Theory of Computing},
    VOLUME = {9},
      YEAR = {1989},
    NUMBER = {4},
     PAGES = {345--362},
      ISSN = {0209-9683},
  MRNUMBER = {1054011},
       DOI = {10.1007/BF02125347},
       URL = {https://doi.org/10.1007/BF02125347},
}

@article {ChungGrahamTournaments,
    AUTHOR = {Chung, Fan R. K. and Graham, Ronald L.},
     TITLE = {Quasi-random tournaments},
   JOURNAL = {J. Graph Theory},
  FJOURNAL = {Journal of Graph Theory},
    VOLUME = {15},
      YEAR = {1991},
    NUMBER = {2},
     PAGES = {173--198},
      ISSN = {0364-9024},
  MRNUMBER = {1106530},
       DOI = {10.1002/jgt.3190150206},
       URL = {https://doi.org/10.1002/jgt.3190150206},
}

@article {ConlonFoxSudakovHereditary,
    AUTHOR = {Conlon, David and Fox, Jacob and Sudakov, Benny},
     TITLE = {Hereditary quasirandomness without regularity},
   JOURNAL = {Math. Proc. Cambridge Philos. Soc.},
  FJOURNAL = {Mathematical Proceedings of the Cambridge Philosophical Society},
    VOLUME = {164},
      YEAR = {2018},
    NUMBER = {3},
     PAGES = {385--399},
      ISSN = {0305-0041},
  MRNUMBER = {3784260},
       DOI = {10.1017/S0305004116001055},
       URL = {https://doi.org/10.1017/S0305004116001055},
}

@article {DiaconisJanson,
    AUTHOR = {Diaconis, Persi and Janson, Svante},
     TITLE = {Graph limits and exchangeable random graphs},
   JOURNAL = {Rend. Mat. Appl. (7)},
  FJOURNAL = {Rendiconti di Matematica e delle sue Applicazioni. Serie VII},
    VOLUME = {28},
      YEAR = {2008},
    NUMBER = {1},
     PAGES = {33--61},
      ISSN = {1120-7183},
  MRNUMBER = {2463439},
}

@article {ElekSzegedyHypergraph,
    AUTHOR = {Elek, G{\'a}bor and Szegedy, Bal{\'a}zs},
     TITLE = {A measure-theoretic approach to the theory of dense hypergraphs},
   JOURNAL = {Adv. Math.},
  FJOURNAL = {Advances in Mathematics},
    VOLUME = {231},
      YEAR = {2012},
    NUMBER = {3-4},
     PAGES = {1731--1772},
      ISSN = {0001-8708},
  MRNUMBER = {2964622},
       DOI = {10.1016/j.aim.2012.06.022},
       URL = {https://doi.org/10.1016/j.aim.2012.06.022},
}

@article {FoxHimwichManiZhou,
    AUTHOR = {Fox, Jacob and Himwich, Zoe and Mani, Nitya and Zhou, Yunkun},
     TITLE = {A note on directed analogues of the {S}idorenko and forcing conjectures},
   JOURNAL = {Electron. J. Combin.},
  FJOURNAL = {The Electronic Journal of Combinatorics},
    VOLUME = {32},
      YEAR = {2025},
    NUMBER = {3},
     PAGES = {Paper No. P3.38, 14},
      ISSN = {1077-8926},
       DOI = {10.37236/11675},
       URL = {https://doi.org/10.37236/11675},
}

@article {GrzesikKralPikhurko,
    AUTHOR = {Grzesik, Andrzej and Kr{\'a}l', Daniel and Pikhurko, Oleg},
     TITLE = {Forcing generalised quasirandom graphs efficiently},
   JOURNAL = {Combin. Probab. Comput.},
  FJOURNAL = {Combinatorics, Probability and Computing},
    VOLUME = {33},
      YEAR = {2024},
    NUMBER = {1},
     PAGES = {16--31},
      ISSN = {0963-5483},
  MRNUMBER = {4680487},
       DOI = {10.1017/S0963548323000263},
       URL = {https://doi.org/10.1017/S0963548323000263},
}

@article {HancockEtAlTournamentForcing,
    AUTHOR = {Hancock, Robert and Kabela, Adam and Kr{\'a}l', Daniel and Martins, Ta{\'i}sa and Parente, Roberto and Skerman, Fiona and Volec, Jan},
     TITLE = {No additional tournaments are quasirandom-forcing},
   JOURNAL = {European J. Combin.},
  FJOURNAL = {European Journal of Combinatorics},
    VOLUME = {108},
      YEAR = {2023},
     PAGES = {Paper No. 103632, 10},
      ISSN = {0195-6698},
  MRNUMBER = {4502205},
       DOI = {10.1016/j.ejc.2022.103632},
       URL = {https://doi.org/10.1016/j.ejc.2022.103632},
}

@article {Hoeffding1948,
    AUTHOR = {Hoeffding, Wassily},
     TITLE = {A class of statistics with asymptotically normal distribution},
   JOURNAL = {Ann. Math. Statistics},
  FJOURNAL = {Annals of Mathematical Statistics},
    VOLUME = {19},
      YEAR = {1948},
    NUMBER = {3},
     PAGES = {293--325},
      ISSN = {0003-4851},
  MRNUMBER = {0026294},
       DOI = {10.1214/aoms/1177730196},
       URL = {https://doi.org/10.1214/aoms/1177730196},
}

@article {HubaiKralParczykPerson,
    AUTHOR = {Hubai, Tam{\'a}s and Kr{\'a}l', Daniel and Parczyk, Olaf and
              Person, Yury},
     TITLE = {More non-bipartite forcing pairs},
   JOURNAL = {Acta Math. Univ. Comenianae},
  FJOURNAL = {Acta Mathematica Universitatis Comenianae. New Series},
    VOLUME = {88},
      YEAR = {2019},
    NUMBER = {3},
     PAGES = {819--825},
      ISSN = {0862-9544},
       URL = {https://www.iam.fmph.uniba.sk/amuc/ojs/index.php/amuc/article/view/1279},
}

@article {JansonQR,
    AUTHOR = {Janson, Svante},
     TITLE = {Quasi-random graphs and graph limits},
   JOURNAL = {European J. Combin.},
  FJOURNAL = {European Journal of Combinatorics},
    VOLUME = {32},
      YEAR = {2011},
    NUMBER = {7},
     PAGES = {1054--1083},
      ISSN = {0195-6698},
  MRNUMBER = {2825535},
       DOI = {10.1016/j.ejc.2011.03.011},
       URL = {https://doi.org/10.1016/j.ejc.2011.03.011},
}

@book {Kallenberg2005,
    AUTHOR = {Kallenberg, Olav},
     TITLE = {Probabilistic Symmetries and Invariance Principles},
    SERIES = {Probability and its Applications (New York)},
 PUBLISHER = {Springer},
   ADDRESS = {New York},
      YEAR = {2005},
     PAGES = {xii+510},
      ISBN = {0-387-25115-4},
  MRNUMBER = {2161313},
       DOI = {10.1007/0-387-28861-9},
       URL = {https://doi.org/10.1007/0-387-28861-9},
}

@book {LovaszBook,
    AUTHOR = {Lov{\'a}sz, L{\'a}szl{\'o}},
     TITLE = {Large Networks and Graph Limits},
    SERIES = {American Mathematical Society Colloquium Publications},
    VOLUME = {60},
 PUBLISHER = {American Mathematical Society},
   ADDRESS = {Providence, RI},
      YEAR = {2012},
     PAGES = {xiv+475},
      ISBN = {978-0-8218-9085-1},
  MRNUMBER = {3012035},
       DOI = {10.1090/coll/060},
       URL = {https://doi.org/10.1090/coll/060},
}

@unpublished {LovaszProblems,
    AUTHOR = {Lov{\'a}sz, L{\'a}szl{\'o}},
     TITLE = {Graph homomorphisms: open problems},
      NOTE = {Manuscript, June 2008},
}

@article {LovaszSos,
    AUTHOR = {Lov{\'a}sz, L{\'a}szl{\'o} and S{\'o}s, Vera T.},
     TITLE = {Generalized quasirandom graphs},
   JOURNAL = {J. Combin. Theory Ser. B},
  FJOURNAL = {Journal of Combinatorial Theory. Series B},
    VOLUME = {98},
      YEAR = {2008},
    NUMBER = {1},
     PAGES = {146--163},
      ISSN = {0095-8956},
  MRNUMBER = {2368030},
       DOI = {10.1016/j.jctb.2007.06.005},
       URL = {https://doi.org/10.1016/j.jctb.2007.06.005},
}

@article {LovaszSzegedy2006,
    AUTHOR = {Lov{\'a}sz, L{\'a}szl{\'o} and Szegedy, Bal{\'a}zs},
     TITLE = {Limits of dense graph sequences},
   JOURNAL = {J. Combin. Theory Ser. B},
  FJOURNAL = {Journal of Combinatorial Theory. Series B},
    VOLUME = {96},
      YEAR = {2006},
    NUMBER = {6},
     PAGES = {933--957},
      ISSN = {0095-8956},
  MRNUMBER = {2274085},
       DOI = {10.1016/j.jctb.2006.05.002},
       URL = {https://doi.org/10.1016/j.jctb.2006.05.002},
}

@article {LovaszSzegedyFF,
    AUTHOR = {Lov{\'a}sz, L{\'a}szl{\'o} and Szegedy, Bal{\'a}zs},
     TITLE = {Finitely forcible graphons},
   JOURNAL = {J. Combin. Theory Ser. B},
  FJOURNAL = {Journal of Combinatorial Theory. Series B},
    VOLUME = {101},
      YEAR = {2011},
    NUMBER = {5},
     PAGES = {269--301},
      ISSN = {0095-8956},
  MRNUMBER = {2802882},
       DOI = {10.1016/j.jctb.2011.03.005},
       URL = {https://doi.org/10.1016/j.jctb.2011.03.005},
}

@article {ReiherSchacht,
    AUTHOR = {Reiher, Christian and Schacht, Mathias},
     TITLE = {Forcing quasirandomness with triangles},
   JOURNAL = {Forum Math. Sigma},
  FJOURNAL = {Forum of Mathematics. Sigma},
    VOLUME = {7},
      YEAR = {2019},
     PAGES = {Paper No. e9, 19},
      ISSN = {2050-5094},
  MRNUMBER = {3933904},
       DOI = {10.1017/fms.2019.7},
       URL = {https://doi.org/10.1017/fms.2019.7},
}

@unpublished {ChatterjeeBhattacharya,
    AUTHOR = {Chatterjee, Sayak and Bhattacharya, Bhaswar B.},
     TITLE = {Transitivity in inhomogeneous random tournaments},
      NOTE = {arXiv:2606.02340},
      YEAR = {2026},
       URL = {https://arxiv.org/abs/2606.02340},
}

@unpublished {KralEtAlRegularTournaments,
    AUTHOR = {Kr{\'a}l', Daniel and Krnc, Matja{\v z} and Ku{\v c}er{\'a}k, Filip and Lidick{\'y}, Bernard and Volec, Jan},
     TITLE = {Sidorenko property and forcing in regular tournaments},
      NOTE = {arXiv:2602.12551},
      YEAR = {2026},
       URL = {https://arxiv.org/abs/2602.12551},
}

@article {NoelRanganathanSimbaqueba,
    AUTHOR = {Noel, Jonathan A. and Ranganathan, Arjun and Simbaqueba, Lina M.},
     TITLE = {Forcing quasirandomness in a regular tournament},
   JOURNAL = {Innov. Graph Theory},
  FJOURNAL = {Innovations in Graph Theory},
    VOLUME = {3},
      YEAR = {2026},
     PAGES = {127--169},
      ISSN = {3050-743X},
       DOI = {10.5802/igt.20},
       URL = {https://doi.org/10.5802/igt.20},
}

@article {SimonovitsSosHereditary,
    AUTHOR = {Simonovits, Mikl{\'o}s and S{\'o}s, Vera T.},
     TITLE = {Hereditarily extended properties, quasi-random graphs and not necessarily induced subgraphs},
   JOURNAL = {Combinatorica},
  FJOURNAL = {Combinatorica. An International Journal on Combinatorics and the Theory of Computing},
    VOLUME = {17},
      YEAR = {1997},
    NUMBER = {4},
     PAGES = {577--596},
      ISSN = {0209-9683},
  MRNUMBER = {1645698},
       DOI = {10.1007/BF01195005},
       URL = {https://doi.org/10.1007/BF01195005},
}

@article {ThornbladTournamentLimits,
    AUTHOR = {Th{\"o}rnblad, Erik},
     TITLE = {Decomposition of tournament limits},
   JOURNAL = {European J. Combin.},
  FJOURNAL = {European Journal of Combinatorics},
    VOLUME = {67},
      YEAR = {2018},
     PAGES = {96--125},
      ISSN = {0195-6698},
  MRNUMBER = {3707221},
       DOI = {10.1016/j.ejc.2017.07.023},
       URL = {https://doi.org/10.1016/j.ejc.2017.07.023},
}

@article {ZhaoHypergraphLimits,
    AUTHOR = {Zhao, Yufei},
     TITLE = {Hypergraph limits: a regularity approach},
   JOURNAL = {Random Structures Algorithms},
  FJOURNAL = {Random Structures \& Algorithms},
    VOLUME = {47},
      YEAR = {2015},
    NUMBER = {2},
     PAGES = {205--226},
      ISSN = {1042-9832},
  MRNUMBER = {3382671},
       DOI = {10.1002/rsa.20537},
       URL = {https://doi.org/10.1002/rsa.20537},
}

\end{document}